\documentclass[reqno]{amsart}
\usepackage{amsfonts,amsmath,amsthm,amssymb,mathrsfs}
\usepackage{color}
\usepackage{amsmath,amssymb,amsfonts,bm}
\usepackage{graphicx}
\usepackage{newunicodechar}
\usepackage{xcolor}
\numberwithin{equation}{section}
\usepackage[pdfstartview=FitH,colorlinks=true]{hyperref}
\hypersetup{urlcolor=red, linkcolor=blue,citecolor=red}
\allowdisplaybreaks

\theoremstyle{plain}

\newtheorem{thm}{Theorem}[section]

\newtheorem{prop}[thm]{Proposition}

\newtheorem{lemma}[thm]{Lemma}

\makeatletter

\newcommand{\Rmnum}[1]{\expandafter\@slowromancap\romannumeral #1@}
\makeatother

\begin{document}
\title[KFP equation]
{Analytic smoothing effect of Cauchy problem for a class of Kolmogorov-Fokker-Planck equations
}

\author[X.-D.Cao \& C.-J. Xu \& Y. Xu]
{Xiao-Dong Cao, Chao-Jiang Xu and Yan Xu}

\address{Xiao-Dong Cao, Chao-Jiang Xu and Yan Xu
\newline\indent
School of Mathematics and Key Laboratory of Mathematical MIIT,
\newline\indent
Nanjing University of Aeronautics and Astronautics, Nanjing 210016, China
}
\email{caoxiaodong@nuaa.edu.cn; xuchaojiang@nuaa.edu.cn; xuyan1@nuaa.edu.cn}

\date{\today}

\subjclass[]{}


\begin{abstract}
We study the Cauchy problem of the Kolmogorov-Fokker-Planck equations and show that the solution enjoys an analytic smoothing effect with $L^2$ initial datum for positive time.
\end{abstract}

\maketitle

\section{Introduction}
In this work, we study the regularity of the Kolmogorov-Fokker-Planck equation (or KFP equations). The Cauchy problem of the KFP equation reads
\begin{equation}\label{1-2}
\begin{cases}
\partial_t u(t, x, y) + y \partial_x u(t, x, y) - \partial_y \big( a \partial_y u(t, x, y) \big) = g(t, x, y) \\
u |_{t = 0} = u_0 (x, y)
\end{cases}
\end{equation}
where $t \ge 0, (x, y) \in \mathbb{R}^2$, $a(t, x, y)\in C^{\infty}([0, \infty[ \times\mathbb R^{2})$ and  there exist $C_{1}, C_{2}>0$ such that 
\begin{equation}\label{6-1}
C_1 \leq a(t, x, y) \leq C_2, \quad \forall x, y\in\mathbb R,\ \  t\ge0.
\end{equation}
The equation \eqref{1-2} is a special case of the following general divergence form KFP equation, 
\begin{equation}\label{8-1}
\mathcal{L}u(t, x) \equiv \sum_{i, j = 1}^{m_0} \partial_{x_i} \left( a_{ij}(t, x) \partial_{x_j} u(t, x) \right) + \sum_{i, j = 1}^n b_{ij} x_i \partial_{x_j} u(t, x) - \partial_t u(t, x) = g(t, x), 
\end{equation}
where $(t, x) = (t, x_1, ..., x_n) \in \mathbb{R}^{n + 1}$, and $1 \leq m_0 \leq n$. The above operator $\mathcal{L}$ which called Kolmogorov-Fokker-Planck type operator (or KFP operator) derives from the theory of diffusion process\cite{p-1}, \cite{p-2} and stochastic theory\cite{p-3}.

There are many studies focusing on the regularity property of the equation \eqref{8-1} and still evolving. Kolmogorov\cite{k-1} in 1934 constructed an explicit fundamental solution of the equation $\partial_x^2 u + x\partial_y u - \partial_t u = 0$ and obtained the smoothness result. Motivated by the work of Kolmogorov\cite{k-1}, H\"ormander\cite{ref7} obtained the far-reaching result \textit{H\"ormander's rank condition} which is a powerful tool when dealing with the KFP operator. One can easily verify that the equation \eqref{1-2} satisfies H\"ormander's rank condition. Lanconelli and Polidoro\cite{s-3} start the studying of the equation \eqref{8-1}. Schauder type estimates of such equation \eqref{8-1} was obtained by Manfredini\cite{s-5}, Francesco and Polidoro\cite{s-6}. By assuming a weak continuity condition on the coefficient $a_{ij}$, the regularity properties of the weak solutions were studied by Bramanti, Cerutti and Manfredini\cite{ref13}, Polidoro and Ragusa\cite{s-4}. By using the Moser iterative scheme introduced by Moser\cite{h-1}, \cite{I-1} for the uniformly parabolic equations, Pascucci and Polidoro\cite{ref2}, Cinti, Pascucci and Polidoro\cite{s-1}, Anceschi, Polidoro and Ragusa\cite{d-5} proved that the weak solutions to certain ultraparabolic equations are locally bounded functions. Besides, there exists other important method when dealing with the regularity of the equation \eqref{8-1}. Zhang\cite{w-1} introduced a weak Poincar\'e type inequality for non-negative weak sub-solutions of homogeneous case of \eqref{8-1} to obtain a local priori estimates, and proved that the weak solution is actually H\"older continuous. Wang and Zhang\cite{ref18} obtained that the weak solution to the non-homogeneous KFP equation is still H\"older continuous by the same technique. The Gevrey regularity results mostly fall on the non-divergence form KFP equation, see\cite{o-2}, \cite{o-3}. For more detail on the classical theory for KFP equation, we can refer to \cite{s-2} and the reference therein.

Let us give the definition of analytic function spaces $\mathcal A(\Omega)$ with $\Omega\subset\mathbb R^{2}$ an open domain. The analytic function spaces $\mathcal A(\Omega)$ is the space of $C^{\infty}$ functions $f$ satisfying:  there exists a constant $C>0$ such that,  for any $\alpha\in\mathbb N^{2}$, 
$$
\left\| \partial^\alpha f \right\|_{L^\infty(\Omega)} \leq C^{\left| \alpha \right| + 1} \alpha !.
$$
Remark that, by using the Sobolev embedding, we can replace the $L^{\infty}$ norm by the $L^{2}$ norm, or the norm in any Sobolev space in the above definition.

In this paper, we consider the Cauchy problem of \eqref{1-2} and would show that the solution of the Cauchy problem to \eqref{1-2} with $L^{2}(\mathbb R^{2})$ initial datum enjoys the analytic regularizing effect. Our main result reads as following.
 
\begin{thm}\label{Theorem}
For any $T > 0$ and $u_0 \in L^2(\mathbb{R}^2)$.  Assume that $a$ satisfies \eqref{6-1}, $a\in C^{\infty}([0, T]; \mathcal{A}(\mathbb{R}^2) ),  g\in C^{\infty}([0, T]; \mathscr{S}(\mathbb{R}^2) )$, and there exists $B, C>0$ such that, for all $\alpha, \beta\in\mathbb N$ and $0\le t\le T$,
\begin{equation}\label{4-1}
\left\| \partial_{x, y}^{\alpha, \beta} a(t) \right\|_{L^\infty \left( \mathbb{R}^2_{x, y} \right)} \leq C^{\alpha + \beta + 1} \alpha! \beta!, \quad \left\| \partial_{x, y}^{\alpha, \beta} g(t) \right\|_{L^2 \left( \mathbb{R}^2_{x, y} \right)} \leq B^{\alpha + \beta + 1} \alpha! \beta!,
\end{equation}
then the Cauchy problem \eqref{1-2} admits an unique solution $u\in C^\infty \left( ]0, T]; \mathcal{A}(\mathbb{R}^2) \right)$. Moreover,  there exists a constant $L > 0$ depending on $C_1, C_2, T, B, C, \delta$, such that
$$
\forall m, n \in \mathbb{N}, \quad \sup_{0 < t \leq T} t^{(\delta + 1) m + \delta n} \| \partial_x^m \partial_y^n u(t) \|_{L^2(\mathbb{R}^2)} \leq L^{m + n + 1} m! n!,
$$
where $\delta>0$ is a sufficiently large real number.
\end{thm}

We have then proved that the Cauchy problem of the degenerate parabolic equation \eqref{1-2} enjoys the same analytic smoothing effect of the following Cauchy problem of parabolic equation: 
\begin{equation*}
\begin{cases}
\partial_t u(t, x, y) - t \partial^2_x u(t, x, y) - \partial^2_y  u(t, x, y) = g(t, x, y), \ \ \ t>0, \ (x, y)\in \mathbb{R}^2,\\
u |_{t = 0} = u_0 (x, y), \ \ \ \   (x, y)\in \mathbb{R}^2,
\end{cases}
\end{equation*}
but with some loss of analytic radius near to $t=0$.  

The paper is organized as following. In section 2, we first give the main commutator estimation. Then, we give the Leibniz formula of the auxiliary operator. In section 3, we study the energy estimation of the Cauchy problem \eqref{1-2}. In section 4, we prove the analytic smoothing effect of the Cauchy problem \eqref{1-2}.

\section{Commutator estimation}

The notation $[T_{1}, T_{2}]=T_{1}T_{2}-T_{2}T_{1}$ is commutator of the operators $T_1, T_2$. We define the following auxiliary vector fields $H_\delta$, 
\begin{equation}\label{1-1}
H_\delta = \frac{t^{\delta + 1}}{\delta + 1} \partial_x + t^\delta \partial_y,
\end{equation}
where $\delta$ is a sufficiently large positive number, for simplifier the notation, we denote $H_\delta$ by $H$.

Now, we give the following significant commutator estimation which has been studied in \cite{ref1}. 
\begin{lemma}(\cite{ref1})\label{lemma2.1}
Assume the operator $H$ defined as \eqref{1-1}, then for any $k\in\mathbb N_{+}$
\begin{equation}\label{1-5}
\left[ \partial_t + y \partial_x, H^k \right] = k \delta t^{\delta - 1} \partial_y H^{k - 1}.
\end{equation}
\end{lemma}

Next, we show the estimation of the operator $H^{k}$. 
\begin{lemma}\label{lemma2.3}
For any $T > 0$ and $k \ge 1$. We have
\begin{equation}\label{6-7}
\left\| H^k a \right\|_{L^\infty([0, T] \times \mathbb{R}^2)} \leq \left( 8C \left( T + 1 \right)^{\delta + 1} \right)^{k + 1} \left( k - 2 \right)!.
\end{equation}
where $a$ is the function that satisfies the condition in Theorem \ref{Theorem}.
\end{lemma}
\begin{proof}
Since \eqref{4-1}, one can obtain
\begin{equation*}
\begin{aligned}
\left\| Ha \right\|_{L^\infty([0, T]\times\mathbb{R}^2)} &\leq \frac{(T + 1)^{\delta + 1}}{\delta + 1}C^2 + (T + 1)^\delta C^2 \leq 2(T + 1)^{\delta + 1}C^2.
\end{aligned}
\end{equation*}
By convention $(-1)!=1$, $0! = 1$, for any $k \ge 2$, by using Leibniz formula, 
\begin{equation*}
\begin{aligned}
\left\| H^k a \right\|_{L^\infty([0, T]\times\mathbb{R}^2)}
&\leq \sum_{i = 0}^k k! \left( T + 1 \right)^{k + k\delta} C^{k + 1} \\
&\leq \left( C \left( T + 1 \right)^{\delta + 1} \right)^{k + 1} \left( k - 2 \right)! \cdot k^2 \cdot \left( k + 1 \right).
\end{aligned}
\end{equation*}
Noting that for any $ k \ge 2$, 
$$
k^2 \cdot \left( k + 1 \right) \leq \left( 2^k \right)^3 \leq 8^k \leq 8^{k + 1},
$$
this implies 
$$
\left\| H^k a \right\|_{L^\infty([0, T] \times\mathbb{R}^2)} \leq \left( 8 C \left( T + 1 \right)^{\delta + 1} \right)^{k + 1} \left( k - 2 \right)! .
$$
\end{proof}

\begin{lemma}\label{lemma2.2}
For any $u\in C^{\infty}([0, T]; \mathscr S(\mathbb R^{2}))$ and $k \ge2$, we have
\begin{equation}\label{H-k}
\begin{split}
     &-(\partial_{y}H^{k}(a\partial_{y}u), H^{k}u)_{L^{2}(\mathbb R^{2})}\geq C_{1}\|\partial_{y}H^{k}u(t)\|^{2}_{L^{2}(\mathbb R^{2})}\\
     &\qquad-\sum_{j=1}^{k}\binom{k}{j}\left( 8C \left( T + 1 \right)^{\delta + 1} \right)^{j + 1} \left( j - 2 \right)!\|\partial_{y}H^{k-j}u(t)\|_{L^{2}(\mathbb R^{2})}\|\partial_{y}H^{k}u(t)\|_{L^{2}(\mathbb R^{2})},
\end{split}     
\end{equation}
where $a$ is the function that satisfies the condition in Theorem \ref{Theorem}.
\end{lemma}
\begin{proof}
We first show the next formula by induction on index $k$,
\begin{equation}\label{2-18}
H^k (a \partial_y u) = \sum_{j = 0}^k \binom{k}{j} \left( H^j a \right) \cdot \left( \partial_y H^{k - j} u \right).
\end{equation}
For $k = 1$, it follows that 
\begin{equation*}
\begin{aligned}
H \left( a \partial_y u \right) &= \left( \frac{t^{\delta + 1}}{\delta + 1} \partial_x + t^\delta \partial_y \right) \cdot \left( a \partial_y u \right) \\
&= \frac{t^{\delta + 1}}{\delta + 1} \left( \partial_x a \right) \cdot \left( \partial_ y u \right) + \frac{t^{\delta + 1}}{\delta + 1}a \cdot\partial_x \partial_y u + t^{\delta} \left( \partial_y a \right) \cdot \left( \partial_y u \right) + t^{\delta} a \cdot \partial_y^2 u \\
&= \left( Ha \right) \cdot \left( \partial_y u \right) + a \cdot \partial_y Hu.
\end{aligned}
\end{equation*}
Assume that for $k\ge2$, and \eqref{2-18} holds for $1\leq m\leq k-1$. Then for $m=k$, by using the induction hypothesis, one has
\begin{equation*}
\begin{aligned}
&H^{k} \left( a \partial_y u \right)=H\left(H^{k-1} \left( a \partial_y u\right) \right)=H\left(\sum_{j=0}^{k-1}\binom{k-1}{j}(H^{j}a)\cdot ( \partial_y H^{k-1 - j} u)\right)\\
&=\sum_{j=1}^{k}\binom{k-1}{j-1}(H^{j}a)\cdot ( \partial_y H^{k- j} u)+H^{k}a\partial_{y}u+\sum_{j=0}^{k-1}\binom{k-1}{j}(H^{j}a)\cdot ( \partial_y H^{k- j} u)\\
&\qquad+a\partial_{y}H^{k}u=\sum_{j = 0}^k \binom{k}{j} \left( H^j a \right) \cdot \left( \partial_y H^{k - j} u \right),
\end{aligned}
\end{equation*}
here we use the fact $\binom{k-1}{j-1}+\binom{k-1}{j}=\binom{k}{j}$.

Thus using integration by parts and \eqref{2-18}, we can get that for all $k\ge1$,
\begin{equation*}
\begin{aligned}
&-\left( \partial_y H^k \left( a \partial_y u \right), H^k u \right)_{L^2\left( \mathbb{R}^2 \right)}
=\left(aH^k \left( \partial_y u \right), \partial_yH^k u \right)_{L^2\left( \mathbb{R}^2 \right)}\\
&\qquad\qquad+\sum_{j=1}^{k}\binom{k}{j}\left(\left( H^j a \right) \cdot \left( \partial_y H^{k - j} u \right), \partial_yH^k u \right)_{L^2\left( \mathbb{R}^2 \right)}.
\end{aligned}
\end{equation*}
Since $a(t, x, y)\geq C_{1}$ for all $x, y\in\mathbb R$ and $t\ge0$ by applying Cauchy-Schwarz inequality, it follows that
\begin{equation*}
\begin{split}
     &-(\partial_{y}H^{k}(a\partial_{y}u), H^{k}u)_{L^{2}(\mathbb R^{2})}\ge C_{1}\|\partial_{y}H^{k}u(t)\|^{2}_{L^{2}(\mathbb R^{2})}\\
     &\qquad-\sum_{j=1}^{k}\binom{k}{j}\left( 8C \left( T + 1 \right)^{\delta + 1} \right)^{j + 1} \left( j - 2 \right)!\|\partial_{y}H^{k-j}u(t)\|_{L^{2}(\mathbb R^{2})}\|\partial_{y}H^{k}u(t)\|_{L^{2}(\mathbb R^{2})}.
\end{split}     
\end{equation*}
\end{proof}

Then, we give the estimations of $H^k g$ in $L^{2}$.

\begin{lemma}\label{lemma2.4}
For any $k \in \mathbb{N}, T > 0$. We have
\begin{equation}\label{6-8}
\left\| H^k g(t) \right\|_{L^2(\mathbb{R}^2)} \leq 2^{k + 1} \left( T + 1 \right)^{\left( k + 1 \right) \left( \delta + 1 \right)} B^{k + 1} k!, \quad \forall t \in [0, T], 
\end{equation}
where $g$ is the function that satisfies the condition in Theorem \ref{Theorem}.
\end{lemma}
\begin{proof}
By using the Leibniz formula, we can get that for all $0 \leq t \leq T$, 
\begin{equation*}
\begin{aligned}
\left\| H^k g(t) \right\|_{L^2(\mathbb{R}^2)} &\leq \left\| \sum_{j = 0}^k \binom{k}{j} \frac{t^{j + k\delta}}{\left( \delta + 1 \right)^j} \partial_x^j \partial_y^{k - j} g(t) \right\|_{L^2(\mathbb{R}^2)} \\
&\leq \sum_{j = 0}^k \binom{k}{j} \frac{t^{j + k\delta}}{\left( \delta + 1 \right)^j} \left\| \partial_x^j \partial_y^{k - j} g(t) \right\|_{L^2(\mathbb{R}^2)}.
\end{aligned}
\end{equation*}
Then from \eqref{4-1}, it follows that for all $0 \leq t \leq T$, 
$$
\left\| H^k g(t) \right\|_{L^2(\mathbb{R}^2)} \leq 2^{k + 1} \left( T + 1 \right)^{\left( k + 1 \right) \left( \delta + 1 \right)} B^{k + 1} k!.
$$
\end{proof}

\section{Energy estimates and smoothing solution}
In this section, we study the energy estimates of the solution to the Cauchy problem \eqref{1-2}. And we prove $u \in C^\infty \left( [0, T], \mathscr{S}(\mathbb{R}^2) \right)$ under the suitable initial datum. For $l > 0$ and $m \in \mathbb{N}$, set 
$$
\left\| h \right\|^2_{H^{m}_l (\mathbb{R}^d)} = \sum_{\left| \alpha \right| \leq m} \int_{\mathbb{R}^d} \left| \left< x \right>^l \partial_x^\alpha h(x) \right|^2 dx.
$$
\begin{prop}\label{lemma3.1}
For any $T > 0$. Let $u_0 \in \mathscr{S}(\mathbb{R}^2)$. Assume $g \in C^\infty \left( [0, T], \mathscr{S} (\mathbb{R}^2) \right)$ and $a$ satisfying \eqref{6-1}, \eqref{4-1}, then the Cauchy problem \eqref{1-2} admits an unique solution $u \in C^\infty \left( [0, T], \mathscr{S} (\mathbb{R}^2) \right)$. Moreover,
for any $t \in [0, T]$, we have
\begin{equation}\label{2-1}
\begin{split}
&\left \| u(t) \right\|^2_{L^2(\mathbb{R}^2)} + \frac{C_1}{2} \int_0^t \left\| \partial_y u(s) \right\|^2_{L^2(\mathbb{R}^2)} ds \\
&\leq T e^T \left\| u_0 \right\|^2_{L^2(\mathbb{R}^2)} + (Te^T+1)\|g\|^{2}_{L^{\infty}([0, T]; L^2(\mathbb{R}^2))}.
\end{split}
\end{equation}
\end{prop}
\begin{proof}
The proof of the existence of weak solution is similar to that in \cite{t-1}, \cite{t-2}. For any $m \in \mathbb{N}, l \ge 0$. Let $P = -\partial_t + \left( y\partial_x - \partial_y \left( a \partial_y \right) \right)^{*}$, here the adjoint operator $\left( \cdot \right)^{*}$ is taken respect to the scalar product in space $H^m_l(\mathbb{R}^2)$. For all $0 \leq t \leq T$ and $f \in C^\infty \left( [0, T], \mathscr{S}(\mathbb{R}^2) \right)$ with $\partial_{x, y}^\alpha f(T) = 0$ for all $\alpha=\left( \alpha_1, \alpha_2 \right) \in \mathbb{N}^2$, one has 
\begin{equation*}
\begin{aligned}
\left( f, Pf \right)_{H^m_l(\mathbb{R}^2)} = -\frac{1}{2} \frac{d}{dt} \left\| f(t) \right\|^2_{H^m_l(\mathbb{R}^2)} + \left( y\partial_x f, f \right)_{H^m_l(\mathbb{R}^2)}  - \left( \partial_y \left( a\partial_y f \right), f \right)_{H^m_l(\mathbb{R}^2)}. 
\end{aligned}
\end{equation*} 
By using the Leibniz formula, we have
\begin{equation*}
\begin{aligned}
&\left( y\partial_x f, f \right)_{H^m_l(\mathbb{R}^2)} = \sum_{\left| \alpha \right| \leq m} \left( \left< \left( x, y \right) \right>^l \partial_x^{\alpha_1} \partial_y^{\alpha_2} (y\partial_x f), \left< \left( x, y \right) \right>^l \partial_{x, y}^\alpha f \right)_{L^2(\mathbb{R}^2)} \\
&=\left( \left< \left( x, y \right) \right>^l y\partial_x(\partial_x^{m}  f), \left< \left( x, y \right) \right>^l \partial_{x}^m f \right)_{L^2(\mathbb{R}^2)}\\
&\quad+\sum_{\alpha_{2}\ne0, |\alpha|\le m}\sum_{\beta\le\alpha_{2}} \binom{\alpha_2}{\beta} \left( \left< \left( x, y \right) \right>^l \partial_x^{\alpha_1}  \left( \left(\partial^{\beta}_{y}y \right) \left(\partial_y^{\alpha_2-\beta}\partial_x f \right) \right), \left< \left( x, y \right) \right>^l \partial_{x, y}^\alpha f \right)_{L^2(\mathbb{R}^2)}.
\end{aligned}
\end{equation*}
Since $\partial^{\beta}_{y}y=0$ for all $\beta\ge2$, then
\begin{equation*}
\begin{aligned}
&\sum_{\alpha_{2}\ne0, |\alpha|\le m} \sum_{\beta\le\alpha_{2}} \binom{\alpha_2}{\beta} \left( \left< \left( x, y \right) \right>^l \partial_x^{\alpha_1}  \left( \left(\partial^{\beta}_{y}y \right) \left(\partial_y^{\alpha_2-\beta}\partial_x f \right) \right), \left< \left( x, y \right) \right>^l \partial_{x, y}^\alpha f \right)_{L^2(\mathbb{R}^2)}\\
&=\sum_{\alpha_{2}\ne0, |\alpha|\le m}\left( \left< \left( x, y \right) \right>^l   y\partial_x(\partial_{x, y}^\alpha f), \left< \left( x, y \right) \right>^l \partial_{x, y}^\alpha f \right)_{L^2(\mathbb{R}^2)}\\
&\quad+\alpha_{2}\sum_{\alpha_{2}\ne0, |\alpha|\le m}\left( \left< \left( x, y \right) \right>^l \partial_x^{\alpha_1+1} \partial_y^{\alpha_2-1}\partial_x f, \left< \left( x, y \right) \right>^l \partial_{x, y}^\alpha f \right)_{L^2(\mathbb{R}^2)}.
\end{aligned}
\end{equation*}
It follows that for all $|\alpha|\le m$
\begin{equation*}
\begin{aligned}
&\left( \left< \left( x, y \right) \right>^l   y\partial_x \left( \partial_{x, y}^\alpha f \right), \left< \left( x, y \right) \right>^l \partial_{x, y}^\alpha f \right)_{L^2(\mathbb{R}^2)}\\
&=\left( y\partial_x \left( \left< \left( x, y \right) \right>^l \partial_{x, y}^\alpha f \right), \left< \left( x, y \right) \right>^l \partial_{x, y}^\alpha f \right)_{L^2(\mathbb{R}^2)}\\
&\quad-\left( y \left( \partial_x\left< \left( x, y \right) \right>^l \right) \left( \partial_{x, y}^\alpha f \right), \left< \left( x, y \right) \right>^l \partial_{x, y}^\alpha f \right)_{L^2(\mathbb{R}^2)}.
\end{aligned}
\end{equation*}
Since   $f \in C^\infty \left( [0, T], \mathscr{S}(\mathbb{R}^2) \right)$, integration by parts,
$$
\left( y\partial_x \left( \left< \left( x, y \right) \right>^l \partial_{x, y}^\alpha f \right), \left< \left( x, y \right) \right>^l \partial_{x, y}^\alpha f \right)_{L^2(\mathbb{R}^2)}=0,
$$
then we can get that 
\begin{equation*}
\begin{aligned}
&\left( y\partial_x f, f \right)_{H^m_l(\mathbb{R}^2)} =-\sum_{|\alpha|\le m}\left( y \left( \partial_x\left< \left( x, y \right) \right>^l \right)  \left( \partial_{x, y}^\alpha f \right), \left< \left( x, y \right) \right>^l \partial_{x, y}^\alpha f \right)_{L^2(\mathbb{R}^2)}\\
&\qquad+\alpha_{2}\sum_{\alpha_{2}\ne0, |\alpha|\le m}\left( \left< \left( x, y \right) \right>^l y\partial_x^{\alpha_1+1} \partial_y^{\alpha_2-1}\partial_x f, \left< \left( x, y \right) \right>^l \partial_{x, y}^\alpha f \right)_{L^2(\mathbb{R}^2)}.
\end{aligned}
\end{equation*}
Note that, 
\begin{equation*}
\begin{aligned}
&|y\partial_x \left< \left( x, y \right) \right>^l| =| l \cdot xy \cdot \left< \left( x, y \right) \right>^{\frac{l}{2} - 1}| \leq l \cdot \frac{x^2 + y^2}{2} \cdot \left< \left( x, y \right) \right>^{\frac{l}{2} - 1} \leq l \cdot \left< \left( x, y \right) \right>^l.
\end{aligned}
\end{equation*}
Then from Cauchy-Schwartz inequality, there exists a constant $C_{3}>0$ such that 
\begin{equation*}
\begin{aligned}
&\left| \sum_{\left| \alpha \right| \leq m} \left( y\partial_x \left< \left( x, y \right) \right>^l \partial_{x, y}^\alpha f, \left< \left( x, y \right) \right>^l \partial_{x, y}^\alpha f \right)_{L^2(\mathbb{R}^2)} \right| \\
&\leq C_3 \sum_{\left| \alpha \right| \leq m} \left\| \left< \left( x, y \right) \right>^l \partial_{x, y}^\alpha f \right\|^2_{L^2(\mathbb{R}^2)} = C_3 \left\| f \right\|^2_{H^m_l(\mathbb{R}^2)},  
\end{aligned}
\end{equation*}
and 
\begin{equation*}
\begin{aligned}
&\alpha_2 \left| \sum_{\alpha_{2}\ne0, \left| \alpha \right| \leq m} \left(  \left< \left( x, y \right) \right>^l \partial_x^{\alpha_1 + 1} \partial_y^{\alpha_2 - 1} f, \left< \left( x, y \right) \right>^l \partial_{x, y}^\alpha f \right)_{L^2(\mathbb{R}^2)} \right| \\
&\leq \sum_{\alpha_{2}\ne0, \left| \alpha \right| \leq m} \alpha_2 \left\| \left< \left( x, y \right) \right>^l \partial_x^{\alpha_1 + 1} \partial_y^{\alpha_2 - 1} f \right\|_{L^2(\mathbb{R}^2)} \left\| \left< \left( x, y \right) \right>^l \partial_{x, y}^\alpha f \right\|_{L^2(\mathbb{R}^2)} \\
&\leq C_3 \left\| f \right\|^2_{H^m_l(\mathbb{R}^2)}.
\end{aligned}
\end{equation*}
Thus we can obtain 
$$
\left| \left( y\partial_x f, f \right)_{H^m_l(\mathbb{R}^2)} \right| \leq 2C_3 \left\| f(t) \right\|^2_{H^m_l(\mathbb{R}^2)}.
$$
By using Leibniz formula, we calculate that
\begin{equation*}
\begin{aligned}
&-\left( \partial_y \left( a\partial_y f \right), f \right)_{H^m_l(\mathbb{R}^2)} = -\sum_{\left| \alpha \right| \leq m} \left( \left< \left( x, y \right) \right>^l \partial_{x, y}^\alpha \partial_y \left( a \partial_y f \right), \left< \left( x, y \right) \right>^l \partial_{x, y}^\alpha f \right)_{L^2(\mathbb{R}^2)} \\
& \qquad = -\sum_{\left| \alpha \right| \leq m} \left( \partial_y \left( \left< \left( x, y \right) \right>^l \partial_{x, y}^\alpha \left( a\partial_y f \right) \right), \left< \left( x, y \right) \right>^l \partial_{x, y}^\alpha f \right)_{L^2(\mathbb{R}^2)} \\
& \qquad \quad + \sum_{\left| \alpha \right| \leq m} \left( \left( \partial_y \left< \left( x, y \right) \right>^l \right) \left( \partial_{x, y}^\alpha \left( a \partial_y f \right) \right), \left< \left( x, y \right) \right>^l \partial_{x, y}^\alpha f \right)_{L^2(\mathbb{R}^2)} \\
& \qquad= \sum_{\left| \alpha \right| \leq m} \left( a \left< \left( x, y \right) \right>^l \partial_{x, y}^\alpha \partial_y f,  \left< \left( x, y \right) \right>^l \partial_{x, y}^\alpha \partial_y f \right)_{L^2(\mathbb{R}^2)} \\
& \qquad \quad + \sum_{\left| \alpha \right| \leq m} \sum_{0 < \beta \leq \alpha} \binom{\alpha}{\beta} \left( \left< \left( x, y \right) \right>^l  \left( \partial_{x, y}^\beta a \right) \left( \partial_{x, y}^{\alpha - \beta} \partial_y f \right), \left< \left( x, y \right) \right>^l \partial_{x, y}^\alpha \partial_y f \right)_{L^2(\mathbb{R}^2)} \\
& \qquad \quad + 2\sum_{\left| \alpha \right| \leq m} \sum_{\beta \leq \alpha} \binom{\alpha}{\beta} \left( \left( \partial_y \left< \left( x, y \right) \right>^l \right) \left( \partial_{x, y}^\beta a \right) \left( \partial_{x, y}^{\alpha - \beta} \partial_y f \right), \left< \left( x, y \right) \right>^l \partial_{x, y}^\alpha f \right)_{L^2(\mathbb{R}^2)}
\end{aligned}
\end{equation*}
Since \eqref{6-1}, it follows that 
\begin{equation*}
\begin{aligned}
&\sum_{\left| \alpha \right| \leq m} \left( a\left< \left( x, y \right) \right>^l \partial_{x, y}^\alpha \partial_y f, \left< \left( x, y \right) \right>^l \partial_{x, y}^\alpha \partial_y f \right)_{L^2(\mathbb{R}^2)} \\
&\ge C_1 \sum_{\left| \alpha \right| \leq m} \left\| \left< \left( x, y \right) \right>^l \partial_{x, y}^\alpha \partial_y f \right\|^2_{L^2(\mathbb{R}^2_{x, y})} = C_1 \left\| \partial_y f \right\|^2_{H^m_l(\mathbb{R}^2)}.
\end{aligned}
\end{equation*}
By using Cauchy-Schwartz inequality, we have
\begin{equation*}
\begin{aligned}
&\left| \sum_{\left| \alpha \right| \leq m} \sum_{0 < \beta \leq \alpha} \binom{\alpha}{\beta} \left( \left< \left( x, y \right) \right>^l  \left( \partial_{x, y}^\beta a \right) \left( \partial_{x, y}^{\alpha - \beta} \partial_y f \right), \left< \left( x, y \right) \right>^l \partial_{x, y}^\alpha \partial_y f \right)_{L^2(\mathbb{R}^2)} \right| \\
&\leq \frac{C_1}{4} \left\| \partial_y f \right\|^2_{H^m_l(\mathbb{R}^2)} \\ 
&\quad+ \frac{1}{C_1} \left( \sum_{\left| \alpha \right| \leq m} \sum_{0 < \beta \leq \alpha} \binom{\alpha}{\beta} \left\| \partial_{x, y}^\beta a \right\|_{L^\infty(\mathbb{R}^2)} \left\| \left< \left( x, y \right) \right>^l \partial_{x, y}^{\alpha - \beta} \partial_y f \right\|_{L^2(\mathbb{R}^2)} \right)^2 \\
&\leq \frac{C_1}{4} \left\| \partial_y f \right\|^2_{H^m_l(\mathbb{R}^2_{x, y})} + \tilde{C}_1 \left\| f \right\|^2_{H^m_l(\mathbb{R}^2)}, 
\end{aligned}
\end{equation*}
and 
\begin{equation*}
\begin{aligned}
&2\left| \sum_{\left| \alpha \right| \leq m} \sum_{\beta \leq \alpha} \binom{\alpha}{\beta} \left( \left( \partial_y \left< \left( x, y \right) \right>^l \right) \left( \partial_{x, y}^\beta a \right) \left( \partial_{x, y}^{\alpha - \beta} \partial_y f \right), \left< \left( x, y \right) \right>^l \partial_{x, y}^\alpha f \right)_{L^2(\mathbb{R}^2)} \right| \\
&\leq C_4 \left\| \partial_y f \right\|_{H^m_l(\mathbb{R}^2)} \left\| f \right\|_{H^m_l(\mathbb{R}^2)} \leq \frac{C_1}{4} \left\| \partial_y f \right\|^2_{H^m_l(\mathbb{R}^2)} + \frac{\left( C_4 \right)^2}{C_1} \left\| f \right\|^2_{H^m_l(\mathbb{R}^2)}.
\end{aligned}
\end{equation*}
Therefore we can get that 
\begin{equation*}
\begin{aligned}
&-\frac{1}{2} \frac{d}{dt} \left\| f(t) \right\|^2_{H^m_l(\mathbb{R}^2)} + \frac{C_1}{2} \left\| \partial_y f(t) \right\|^2_{H^m_l(\mathbb{R}^2)}  \\
&\qquad- \left( 2C_3 + \tilde{C}_1 + \frac{\left( C_4 \right)^2}{C_1} \right) \left\| f(t) \right\|^2_{H^m_l(\mathbb{R}^2)}\leq \left\| f(t) \right\|_{H^m_l(\mathbb{R}^2)} \left\| Pf(t) \right\|_{H^m_l(\mathbb{R}^2)},
\end{aligned}
\end{equation*}
which implies that 
\begin{equation*}
\begin{aligned}
&-\frac{d}{dt} \left(e^{c_0 t} \left\| f(t) \right\|^2_{H^m_l(\mathbb{R}^2)} \right) + C_1 e^{c_0 t} \left\| \partial_y f(t) \right\|^2_{H^m_l(\mathbb{R}^2)} \\
&\leq 2e^{c_0 t} \left\| f(t) \right\|_{H^m_l(\mathbb{R}^2)} \left\| Pf(t) \right\|_{H^m_l(\mathbb{R}^2)},
\end{aligned}
\end{equation*}
with $c_0 = 4C_3 + 2\tilde{C}_1 + \frac{2\left( C_4 \right)^2}{C_{1}}$.
Since $\partial_{x, y}^\alpha f(T) = 0$ for all $\alpha \in \mathbb{N}^2$, we have
\begin{equation*}
\begin{aligned}
&\left\| f(t) \right\|^2_{H^m_l(\mathbb{R}^2)} + \int_t^T e^{c_0 (s - t)} \left\| \partial_y f(s) \right\|^2_{H^m_l(\mathbb{R}^2)} ds \\
&\leq 2\int_t^T e^{c_0 (s - t)} \left\| f(s) \right\|_{H^m_l(\mathbb{R}^2)} \left\| Pf(s) \right\|_{H^m_l(\mathbb{R}^2)} ds \\
&\leq 2e^{c_0 T} \left\| f \right\|_{L^\infty \left( [0, T], H^m_l(\mathbb{R}^2) \right)} \left\| Pf \right\|_{L^1 \left( [0, T], H^m_l(\mathbb{R}^2) \right)}, 
\end{aligned}
\end{equation*}
which leading to
\begin{equation}\label{5-10}
\begin{aligned}
\left\| f \right\|_{L^\infty \left( [0, T], H^m_l(\mathbb{R}^2) \right)} \leq 2e^{c_0 T} \left\| Pf \right\|_{L^1 \left( [0, T], H^m_l(\mathbb{R}^2) \right)}.
\end{aligned}
\end{equation}
For all $m \in \mathbb{N}, l \ge 0$, consider the subspace
$$
\Omega = \left\{ Pf: f \in C^\infty \left( [0, T], \mathscr{S}(\mathbb{R}^2) \right), \partial_{x, y}^\alpha f(T) = 0, \forall \alpha \in \mathbb{N}^2 \right\} \subset L^1 \left( [0, T], H^m_l(\mathbb{R}^2) \right).
$$
Define the linear functional
\begin{equation*}
\begin{aligned}
&\mathcal{G}: \Omega \rightarrow \mathbb{R} \\
&\omega=Pf \rightarrow \left( u_0, f(0) \right)_{H^m_l(\mathbb{R}^2)} + \int_0^T \left( g(t), f(t) \right)_{H^m_l(\mathbb{R}^2)} dt.
\end{aligned}
\end{equation*}
Let $Pf_1 = Pf_2$ with $f_1, f_2 \in C^\infty \left( \left[ 0, T \right], \mathscr{S}(\mathbb{R}^2) \right)$ and $f_1 (T) = f_2 (T) = 0$, then from \eqref{5-10}, it follows that
$$
\left\| f_1 - f_2 \right\|_{L^\infty \left( \left[0, T \right], H^m_l(\mathbb{R}^2) \right)} \leq 2e^{c_0 T} \left\| Pf_1 - Pf_2 \right\|_{L^1 \left( [0, T], H^m_l(\mathbb{R}^2) \right)} = 0,
$$
hence $f_1 = f_2$, which leading to the fact that the operator $P$ is injective. Therefore the linear functional $\mathcal{G}$ is well-defined. Moreover, we can obtain 
\begin{equation*}
\begin{aligned}
\left| \mathcal{G}\left( Pf \right) \right| 
&\leq \left( \left\| u_0 \right\|_{H^m_l(\mathbb{R}^2)} + \int_0^T \left\| g(t) \right\|_{H^m_l(\mathbb{R}^2)} dt \right) \cdot \left\| f \right\|_{L^\infty \left( [0, T], H^m_l(\mathbb{R}^2) \right)} \\
&\leq 2e^{c_0 T} \left( \left\| u_0 \right\|_{H^m_l(\mathbb{R}^2)} + \int_0^T \left\| g(t) \right\|_{H^m_l(\mathbb{R}^2)} dt \right) \cdot \left\| Pf \right\|_{L^1 \left( [0, T], H^m_l(\mathbb{R}^2) \right)}.
\end{aligned}
\end{equation*}
Hence $\mathcal{G}$ is a continuous linear functional on $\left( \Omega, \left\| \cdot \right\|_{L^1 \left( [0, T], H^m_l(\mathbb{R}^2) \right)} \right)$. By using the Hahn-Banach theorem, $\mathcal{G}$ can be extended to a continuous linear functional on $L^1 \left( [0, T], H^m_l(\mathbb{R}^2) \right)$ with norm can be bounded by 
$$2e^{c_0 T} \left( \left\| u_0 \right\|_{H^m_l(\mathbb{R}^2)} + \int_0^T \left\| g(t) \right\|_{H^m_l(\mathbb{R}^2)} dt \right).$$
It follows the Riesz representation theorem that there exists a unique function $u \in L^\infty \left( [0, T], H^m_l(\mathbb{R}^2) \right)$ satisfying 
$$
\mathcal{G}(\omega) = \int_0^T \left( u, \omega \right)_{H^m_l(\mathbb{R}^2)} dt, \forall \omega \in \Omega, 
$$
and
$$
\left\| u \right\|_{L^\infty \left( [0, T], H^m_l(\mathbb{R}^2) \right)} \leq 2e^{c_0 T} \left( \left\| u_0 \right\|_{H^m_l(\mathbb{R}^2)} + \left\| g \right\|_{L^1 \left( [0, T], H^m_l(\mathbb{R}^2) \right)} \right),
$$
these imply that for all $f \in C^\infty \left( [0, T], \mathscr{S}(\mathbb{R}^2) \right)$, 
\begin{equation*}
\begin{aligned}
\mathcal{G} \left( Pf \right) &= \int_0^T \left( u, Pf \right)_{H^m_l(\mathbb{R}^2)} dt = \left( u_0, f(0) \right)_{H^m_l(\mathbb{R}^2)} + \int_0^T \left( g(t), f(t) \right)_{H^m_l(\mathbb{R}^2)} dt.
\end{aligned}
\end{equation*}
Therefore, $\forall m \in \mathbb{N}, l \ge 0$, $u \in L^\infty \left( [0, T], H^m_l(\mathbb{R}^2) \right)$ is a unique weak solution of the Cauchy problem \eqref{1-2}. Next, we would show by induction that for all integers $k$,
\begin{equation}\label{5-4}
\frac{d^k}{dt^k} u(t) \in L^\infty ([0, T], \mathscr{S}(\mathbb{R}^2)).
\end{equation}
For $k = 0$, we have already proved. Assume \eqref{5-4} holds for k, then we try to prove
$$
\frac{d^{k + 1}}{dt^{k + 1}} u(t) \in L^\infty ([0, T], \mathscr{S}(\mathbb{R}^2)).
$$
Since u is the solution of the Cauchy problem \eqref{1-2}, we have
$$
\frac{d^{k + 1}}{dt^{k + 1}} u(t)  = \partial_t^k g - y \partial_x \partial_t^k u + \partial_t^{k} \partial_y \left( a \partial_y u \right).
$$
By using the induction hypothesis, it follows that
$$
y \partial_x \partial_t^k u,\  \partial_t^{k} \partial_y \left( a \partial_y u \right) \in L^\infty ([0, T], \mathscr{S}(\mathbb{R}^2)), 
$$
since $g \in C^\infty([0, T], \mathscr{S}(\mathbb{R}^2_{x, y}))$, we have $\partial_t^k g \in L^\infty ([0, T], \mathscr{S}(\mathbb{R}^2))$. So that 
$$
\frac{d^{k + 1}}{dt^{k + 1}} u(t) \in L^\infty ([0, T], \mathscr{S}(\mathbb{R}^2)).
$$
Hence, we can obtain that
$$
u \in C^\infty \left( [0, T], \mathscr{S}(\mathbb{R}^2) \right).
$$

It remains to consider \eqref{2-1}. Since $u \in C^\infty([0, T], \mathscr{S}(\mathbb{R}^2))$ is the solution of the Cauchy problem \eqref{1-2}, it follows that 
$$
\frac{1}{2} \frac{d}{dt} \left\| u(t) \right\|^2_{L^2(\mathbb{R}^2)}+(y\partial_{x}u, u)_{L^2(\mathbb{R}^2)}-(\partial_{y}(a \partial_y u), u)_{L^2(\mathbb{R}^2)}= \left( g, u \right)_{L^2(\mathbb{R}^2)}.
$$
Integration by parts, one has
$$(y\partial_{x}u, u)_{L^2(\mathbb{R}^2)}=0,$$
since $a(t, x, y)\geq C_{1}$ for all $x, y\in\mathbb R$ and $t\ge0$, one has
$$-(\partial_{y}(a\partial_{y}u), u)_{L^2(\mathbb{R}^2)}=(a\partial_{y}u, \partial_{y}u)_{L^2(\mathbb{R}^2)}\geq C_{1}\|\partial_{y}u\|^{2}_{L^2(\mathbb{R}^2)}.$$
Thus by the Cauchy-Schwartz inequality, we obtain
\begin{equation*}
\frac{1}{2} \frac{d}{dt} \left\| u(t) \right\|^2_{L^2(\mathbb{R}^2)} + C_1 \left\| \partial_y u(t) \right\|^2_{L^2(\mathbb{R}^2)} \leq \frac{1}{2}\|g(t)\|^{2}_{L^2(\mathbb{R}^2)}+ \frac{1}{2} \left\| u(t) \right\|^2_{L^2(\mathbb{R}^2)}.
\end{equation*}
By using Gronwall's inequality, one has
$$
\left\| u(t) \right\|^2_{L^2(\mathbb{R}^2)} \leq e^T \left( \left\| u_0 \right\|^2_{L^2(\mathbb{R}^2)} + T\|g\|^{2}_{L^{\infty}([0, T]; L^2(\mathbb{R}^2))}\right),
$$
and therefore, we have
\begin{equation*}
\begin{split}
&\frac{1}{2} \frac{d}{dt} \left\| u(t) \right\|^2_{L^2(\mathbb{R}^2)} + C_1 \left\| \partial_y u(t) \right\|^2_{L^2(\mathbb{R}^2)}\\
&\leq \frac{1}{2}\|g(t)\|^{2}_{L^2(\mathbb{R}^2)}+ \frac{1}{2} e^T \left\| u_0 \right\|^2_{L^2(\mathbb{R}^2)} + \frac{1}{2} Te^T \|g\|^{2}_{L^{\infty}([0, T]; L^2(\mathbb{R}^2))}.
\end{split}
\end{equation*}
For all $0\leq t \leq T$, integrate the above inequality from~0~to $t$, 
\begin{equation*}
\begin{split}
&\left\| u(t) \right\|^2_{L^2(\mathbb{R}^2)} + \frac{C_1}{2} \int_0^t \left\| \partial_y u(s) \right\|^2_{L^2(\mathbb{R}^2)} ds\\
&\leq \left( 1 + T e^T \right) \left\| u_0 \right\|^2_{L^2(\mathbb{R}^2)} + (T^2e^T + T)\|g\|^{2}_{L^{\infty}([0, T]; L^2(\mathbb{R}^2))}.
\end{split}
\end{equation*}
This energy estimate imply also the unicity of solutions for $L^2$ initial datum. 
\end{proof}

\section{Proof of theorem \ref{Theorem}}
In this section, we show the analytic regularity of the solution. We first construct the following estimation for auxiliary vector fields $H=H_\delta$.

\begin{prop}\label{Proposition 4.1.}
For any $k \in \mathbb{N}$. Let $u$ be the solution of the Cauchy problem \eqref{1-2}, then there exists $A > 0$ which independent of $k$, such that for any $t \in \left] 0, T \right]$, 
\begin{equation}\label{2-2}
\left\| H^k u(t) \right\|^2_{L^2(\mathbb{R}^2)} + \frac{C_1}{2} \int_0^t \left\| \partial_y H^k u(s) \right\|^2_{L^2(\mathbb{R}^2)} ds \leq \left( A^{k + 1} k! \right)^2,
\end{equation}
where $A$ depends on $T$ and $\delta$.
\end{prop}
\begin{proof}
Fixing $\varphi\ge0$,  a smooth function of compact support, defined in $\mathbb R_{x}\times\mathbb R_{y}$, with the properties that for $0<\epsilon<1$, $\varphi_{\epsilon}(x, y)=\epsilon^{-2}\varphi(\frac{x}{\epsilon}, \frac{y}{\epsilon})$ and $\int_{\mathbb{R}^2_{x, y}} \varphi_{\epsilon}(x, y) dxdy = 1$. And we set $u_{0, \epsilon}=(u_{0}*\varphi_{\epsilon})e^{-\epsilon(x^{2}+y^{2})}$. Since $u_{0}\in L^{2}(\mathbb R^{2})$, from the Leibniz formula, one has for all $ \alpha\in\mathbb N^{2}$ and $l\ge0$
$$\langle (x, y)\rangle^{l}\partial^{\alpha}_{x, y}u_{0, \epsilon}=\sum_{\beta\le\alpha}\binom{\alpha}{\beta} \langle (x, y)\rangle^{l}\left(u_{0}*\partial^{\beta}_{x, y}\varphi_{\epsilon}\right)\partial^{\alpha-\beta}_{x, y}e^{-\epsilon(x^{2}+y^{2})},$$
then from Minkowski inequality and Young inequality, we can obtain that for all $m \in \mathbb N$ and $l\ge0$
\begin{equation*}\label{u-delta}
\begin{split}
     \left\|u_{0, \epsilon}\right\|^{2}_{H^{m}_{l}(\mathbb R^{2})}
     &\le\sum_{|\alpha|\le m}\bigg(\sum_{\beta\le\alpha} \binom{\alpha}{\beta} \left\|\langle (x, y)\rangle^{l}\partial^{\alpha-\beta}_{x, y}e^{-\epsilon(x^{2}+y^{2})}\right\|_{L^{\infty}(\mathbb R^{2})}\\
     &\qquad\times\|\partial^{\beta}_{x, y}\varphi_{\epsilon}\|_{L^{1}(\mathbb R^{2})}\|u_{0}\|_{L^{2}(\mathbb R^{2})}\bigg)^{2}\le C_{\epsilon, m, l}\|u_{0}\|^{2}_{L^{2}(\mathbb R^{2})},
\end{split}
\end{equation*}
so that $u_{0, \epsilon}\in \mathscr{S}(\mathbb{R}^2)$. Also by using Minkowski inequality, we have 
\begin{equation}\label{u-epsilon}
\begin{split}
    \left\|u_{0, \epsilon}\right\|_{L^{2}(\mathbb R^{2})}&\le\left(\int_{\mathbb R^{2}}\left|\int_{\mathbb R^{2}}u_{0}(x-x', y-y')\varphi_{\epsilon}(x ', y ')dx'dy'\right|^{2}dxdy\right)^{\frac12}\\
    &\le\int_{\mathbb R^{2}}\left(\int_{\mathbb R^{2}}\left|u_{0}(x-x', y-y')\right|^{2}dxdy\right)^{\frac12}\varphi_{\epsilon}(x ', y ')dx ' d y '\\
    &=\left\|u_{0}\right\|_{L^{2}(\mathbb R^{2})}\int_{\mathbb R^{2}}\varphi_{\epsilon}(x ', y ')dx ' d y '=\left\|u_{0}\right\|_{L^{2}(\mathbb R^{2})}.
\end{split}
\end{equation}
Using Proposition \ref{lemma3.1}, the following Cauchy problem, for any $0<\epsilon<1$, 
\begin{equation}\label{1-2A}
\begin{cases}
\partial_t u_{\epsilon}(t, x, y) + y \partial_x u_{\epsilon}(t, x, y) - \partial_y \big( a \partial_y u_{\epsilon}(t, x, y) \big) = g(t, x, y), \\
u_{\epsilon} |_{t = 0} = u_{0, \epsilon} (x, y),
\end{cases}
\end{equation}
admits a unique solution $u_{\epsilon}\in C^{\infty}([0, T]; \mathscr{S}(\mathbb{R}^2))$, and from \eqref{2-1} and \eqref{u-epsilon}, one has for all $0\le t\le T$
\begin{equation*}
\begin{split}
&\left \| u_{\epsilon}(t) \right\|^2_{L^2(\mathbb{R}^2)} + \frac{C_1}{2} \int_0^t \left\| \partial_y u_{\epsilon}(s) \right\|^2_{L^2(\mathbb{R}^2)} ds\le(A_{0})^{2},
\end{split}
\end{equation*}
here
$$A_{0}=\sqrt{\left( 1 + T e^T \right) \left\| u_0 \right\|^2_{L^2(\mathbb{R}^2)} + (T^2e^T + T)\|g\|^{2}_{L^{\infty}([0, T]; L^2(\mathbb{R}^2))}},$$
which is independent of $\epsilon>0$. Now, we would prove by induction on the index $k$ that
$$
\left\| H^k u_{\epsilon}(t) \right\|^2_{L^2(\mathbb{R}^2)} + \frac{C_1}{2} \int_0^t \left\| \partial_y H^k u_{\epsilon}(s) \right\|^2_{L^2(\mathbb{R}^2)} ds \leq \left( A^{k + 1} k! \right)^2
$$
with $A>0$ independent of $ \epsilon$ and $k$.  
For $k=0$, it has been proved. Assume that $k\ge1$ and for all $0\le m\le k-1$ 
\begin{equation}\label{induction hypothesis}
     \left\| H^m u_{\epsilon}(t) \right\|^2_{L^2(\mathbb{R}^2)} + \frac{C_1}{2} \int_0^t \left\| \partial_y H^m u_{\epsilon}(t) \right\|^2_{L^2(\mathbb{R}^2)} ds \leq \left( A^{m + 1} m! \right)^2, \ \forall t\in]0, T].
\end{equation}

Now, we would to show that \eqref{induction hypothesis} holds true for $m=k$. Since $u_{\epsilon} \in C^{\infty}(]0, T]; \mathscr{S}(\mathbb{R}^2))$ is the solution of Cauchy problem \eqref{1-2A}, we have 
\begin{equation}\label{2-3}
\begin{split}
&\frac{1}{2} \frac{d}{dt} \left\| H^k u_{\epsilon}(t) \right\|^2_{L^2(\mathbb{R}^2)} - \left( H^k \partial_y \left( a \partial_y u_{\epsilon} \right), H^k u_{\epsilon} \right)_{L^2(\mathbb{R}^2)} \\
&= \left( H^k g, H^k u_{\epsilon} \right)_{L^2(\mathbb{R}^2)} + \left( \left[ \partial_t + y \partial_x, H^k \right] u_{\epsilon}, H^k u_{\epsilon} \right)_{L^2(\mathbb{R}^2)}.
\end{split}
\end{equation}
Then from Lemma \ref{lemma2.1} and Lemma \ref{lemma2.2}, it follows that 
\begin{equation}\label{2-3}
\begin{split}
&\frac{1}{2} \frac{d}{dt} \left\| H^k u_{\epsilon}(t) \right\|^2_{L^2(\mathbb{R}^2)}+C_{1}\|\partial_{y}H^{k}u_{\epsilon}\|^{2}_{L^{2}(\mathbb R^{2})}\\
&\le\left|\left( H^k g, H^k u_{\epsilon} \right)_{L^2(\mathbb{R}^2)}\right| + k\left|\left( \delta t^{\delta - 1} \partial_y H^{k - 1}u_{\epsilon}, H^k u_{\epsilon} \right)_{L^2(\mathbb{R}^2)}\right|\\
&\quad+\sum_{j=1}^{k}\binom{k}{j}\left( 8C \left( T + 1 \right)^{\delta + 1} \right)^{j + 1} \left( j - 2 \right)!\|\partial_{y}H^{k-j}u_{\epsilon}\|_{L^{2}(\mathbb R^{2})}\|\partial_{y}H^{k}u_{\epsilon}\|_{L^{2}(\mathbb R^{2})}\\
&=I_{1}+I_{2}+I_{3}.
\end{split}
\end{equation}
For the terms $I_1$ and $I_{2}$, applying the Cauchy-Schwarz inequality, we have
\begin{equation*}
\begin{aligned}
I_1 
&\leq\left\| H^k g(t) \right\|_{L^2(\mathbb{R}^2)} \left\| H^k u_{\epsilon}(t)  \right\|_{L^2(\mathbb{R}^2)} \leq 2 \left\| H^k g(t)  \right\|^2_{L^2(\mathbb{R}^2)} + \frac{1}{8} \left\| H^k u_{\epsilon} (t) \right\|^2_{L^2(\mathbb{R}^2)},
\end{aligned}
\end{equation*}
and
\begin{equation*}
\begin{aligned}
I_2 
&\le 2\left( k\delta t^{\delta - 1} \right)^2 \left\| \partial_y H^{k - 1} u_{\epsilon}(t) \right\|^2_{L^2(\mathbb{R}^2)} + \frac{1}{8} \left\| H^k u_{\epsilon}(t) \right\|^2_{L^2(\mathbb{R}^2)}.
\end{aligned}
\end{equation*}
For the term $I_3$, by using the Cauchy-Schwarz inequality, it follows that 
\begin{equation*}
\begin{aligned}
I_3 &\leq \frac{2}{C_1} \left( \sum_{j = 1}^k \binom{k}{j} \left( 8C \left( T + 1 \right)^{\delta + 1} \right)^{j + 1} \left( j - 2 \right)! \left\| \partial_y H^{k - j} u_{\epsilon}(t) \right\|_{L^2(\mathbb{R}^2)} \right)^2\\
&\qquad\quad + \frac{C_1}{8} \left\| \partial_y H^k u_{\epsilon}(t) \right\|^2_{L^2(\mathbb{R}^2)}.
\end{aligned}
\end{equation*}
Combining $I_{1}-I_{3}$ shows that
\begin{equation*}
\begin{aligned}
&\frac{d}{dt} \left\| H^k u_{\epsilon}(t) \right\|^2_{L^2(\mathbb{R}^2)} + \frac74C_1 \left\| \partial_y H^k u_{\epsilon}(t) \right\|^2_{L^2(\mathbb{R}^2)} \\
&\leq 4\left\| H^k g(t) \right\|^2_{L^2(\mathbb{R}^2)} + 4 \left( k\delta t^{\delta - 1} \right)^2 \left\| \partial_y H^{k - 1} u_{\epsilon}(t) \right\|^2_{L^2(\mathbb{R}^2)} + \frac{1}{2} \left\| H^k u_{\epsilon}(t) \right\|^2_{L^2(\mathbb{R}^2)}\\
&\quad + \frac{4}{C_1} \left( \sum_{j = 1}^k \binom{k}{j} \left( 8C \left( T + 1 \right)^{\delta + 1} \right)^{j + 1} \left( j - 2 \right)! \left\| \partial_y H^{k - j} u_{\epsilon}(t) \right\|_{L^2(\mathbb{R}^2)} \right)^2.
\end{aligned}
\end{equation*}
Thus for any $t \in [0, T]$, integrate from~0~to $t$, we can obtain that
\begin{equation}\label{2-5}
\begin{aligned}
&\left\| H^k u_{\epsilon}(t) \right\|^2_{L^2(\mathbb{R}^2)} + \frac{7}{4} C_1 \int_0^t \left\| \partial_y H^k u_{\epsilon}(s) \right\|^2_{L^2(\mathbb{R}^2)} ds \\
&\leq 4\int_0^t \left\| H^k g(s) \right\|^2_{L^2(\mathbb{R}^2)} ds + 4 \int_0^t \left( k\delta s^{\delta - 1} \right)^2 \left\| \partial_y H^{k - 1} u_{\epsilon}(s) \right\|^2_{L^2(\mathbb{R}^2)} ds \\
& \quad + \frac{4}{C_1} \int_0^t \left( \sum_{j = 1}^k \binom{k}{j} \left( 8C \left( T + 1 \right)^{\delta + 1} \right)^{j + 1} \left( j - 2 \right)! \left\| \partial_y H^{k - j} u_{\epsilon}(s) \right\|_{L^2(\mathbb{R}^2)} \right)^2 ds \\
& \quad + \left\| H^k u_{\epsilon}(t) \right\|^2_{L^2(\mathbb{R}^2)} \bigg{|}_{t = 0} + \frac{1}{2} \int_0^t \left\| H^k u_{\epsilon}(s) \right\|^2_{L^2(\mathbb{R}^2)} \\
&= R_1 + R_2 + R_3 + R_4 + \frac{1}{2} \int_0^t \left\| H^k u_{\epsilon}(s) \right\|^2_{L^2(\mathbb{R}^2)} .
\end{aligned}
\end{equation}
For the term $R_1$, from lemma \ref{lemma2.4}, we obtain that
\begin{equation*}
\begin{aligned}
R_{1}
&\leq 4T\left(2^{k + 1}  \left( T + 1 \right)^{(k + 1)(\delta + 1)}B^{k + 1} k! \right)^2,
\end{aligned}
\end{equation*}
taking $A\ge \left( 2 \left( T + 1 \right)^{\delta + 1}B \right)^2 + 1$, and note that $k + 1 \leq 2k$, thus we can deduce 
\begin{equation}\label{2-6}
\begin{aligned}
R_{1}
&\leq 4T\left(A^{k} k! \right)^2.
\end{aligned}
\end{equation}
For the term $R_2$, since $\delta>0$ large enough, from the induction hypothesis \eqref{induction hypothesis}, it follows that for all $0\leq t\leq T$
\begin{equation}\label{2-7}
\begin{aligned}
R_2 
&\leq 4k^2 \delta^2 \left( T + 1 \right)^{2\left( \delta - 1 \right)} \int_0^t \left\| \partial_y H^{k - 1} u_{\epsilon} \right\|^2_{L^2(\mathbb{R}^2)} ds \\
&\leq \frac{8}{C_1} \delta^2 \left( T + 1 \right)^{2\left( \delta + 1 \right)}  \left( A^k k! \right)^2.
\end{aligned}
\end{equation}
For $R_3$, by using the Minkowski inequality and the induction hypothesis \eqref{induction hypothesis}, it follows that 
\begin{equation}\label{R-3}
\begin{aligned}
R_3 
&\leq \frac{4}{C_1} \left( \sum_{j = 1}^k \binom{k}{j} \left( 8C \left( T + 1 \right)^{\delta + 1} \right)^{j + 1} \left( j - 2 \right)! \left( \int_0^t \left\| \partial_y H^{k - j} u_{\epsilon}(s) \right\|^2_{L^2(\mathbb{R}^2)} ds \right)^{\frac{1}{2}} \right)^2 \\
&\leq \frac{4}{C_1} \left( \sum_{j = 1}^k \binom{k}{j} \left( 8C \left( T + 1 \right)^{\delta + 1} \right)^{j + 1} \left( j - 2 \right)! \sqrt{\frac{2}{C_1}}  A^{k - j + 1}(k - j)! \right)^2\\
&\leq \frac{8\left( k! \right)^2}{\left( C_1 \right)^2} \left( \sum_{j = 2}^k \frac{1}{j \left( j - 1 \right)} (8C(T + 1)^{\delta + 1})^{j + 1} A^{k - j +1} + 2 \left( T + 1 \right)^{\delta + 1} C^2A^k \right)^2,
\end{aligned}
\end{equation}
then taking $A \ge (8C(T + 1)^{\delta + 1})^3$, since $j + 1 \leq 3\left( j - 1 \right), \forall j\geq2$, we have
\begin{equation*}
\begin{aligned}
R_3 
&\leq \frac{8}{\left( C_1 \right)^2} \left( 2\left( T + 1 \right)^{\delta + 1} C^2 A^k + \sum_{j = 2}^k \frac{1}{j(j - 1)} A^k \right)^2 \left( k! \right)^2.
\end{aligned}
\end{equation*}
Noting that
$$
\sum_{j = 2}^k \frac{1}{j \left( j - 1 \right)} = 1 - \frac{1}{k} \leq 1.
$$
Thus we have
\begin{equation}\label{2-8}
R_3 \leq \frac{8}{\left( C_1 \right)^2} \left( 1 +  2\left( T + 1 \right)^{\delta + 1} C^2 \right)^2 \left( A^k k! \right)^2.
\end{equation}

It remains to consider the term $R_4$. By using the Leibniz formula, it follows that for any $t\in[0, T]$
\begin{equation*}
\begin{aligned}
\left\| H^k u_{\epsilon}(t) \right\|_{L^2(\mathbb{R}^2)} &\leq \sum_{j = 0}^k \binom{k}{j} t^{j + k\delta} \sup_{0 \leq t \leq T} \left\| \partial_x^j \partial_y^{k - j} u_{\epsilon}(t) \right\|_{L^2(\mathbb{R}^2)} \\
&\leq t^{k (\delta+1)} \max_{0 \leq j \leq k} \binom{k}{j} \sum_{j = 0}^k \sup_{0 \leq t \leq T} \left\| \partial_x^j \partial_y^{k - j} u_{\epsilon}(t) \right\|_{L^2(\mathbb{R}^2)}.
\end{aligned}
\end{equation*}
Since $u_{\epsilon}\in C^{\infty}([0, T]; \mathscr S(\mathbb R^{2}))$, it follows that 
$$\sup_{0 \leq t \leq T} \left\| \partial_x^j \partial_y^{k - j} u_{\epsilon}(t) \right\|_{L^2(\mathbb{R}^2)}<\infty,$$
and therefore we have, for $k\ge 1$, 
\begin{equation}\label{2-9}
\left\| H^k u_{\epsilon}(t) \right\|^2_{L^2(\mathbb{R}^2)} \bigg|_{t = 0} = 0.
\end{equation}
Combining \eqref{2-6}, \eqref{2-7}, \eqref{2-8}, \eqref{2-9}, it follows that 
\begin{equation}\label{2-10}
\begin{aligned}
& \quad \left\| H^k u_{\epsilon}(t) \right\|^2_{L^2(\mathbb{R}^2)} + \frac{7}{4} C_1 \int_0^t \left\| \partial_y H^k u_{\epsilon}(s) \right\|^2_{L^2(\mathbb{R}^2)} ds \\
&\leq 4T\left(A^{k} k! \right)^2 + \frac{8}{C_1} \delta^2 \left( T + 1 \right)^{2\left( \delta + 1 \right)} \left( A^k k! \right)^2 \\
&\quad + \frac{8}{\left( C_1 \right)^2} \left( 1 + 2\left( T + 1 \right)^{\delta + 1} C^2 \right)^2 \left( A^k k! \right)^2 + \frac{1}{2} \int_0^t \left\| H^k u_{\epsilon}(s) \right\|^2_{L^2(\mathbb{R}^2)} ds\\
&= C_{5}\left( A^k k! \right)^2+\frac{1}{2} \int_0^t \left\| H^k u_{\epsilon}(s) \right\|^2_{L^2(\mathbb{R}^2)} ds,
\end{aligned}
\end{equation}
where
$$C_{5}=4T + \frac{8}{C_1} \delta^2 \left( T + 1 \right)^{2\left( \delta + 1 \right)} + \frac{8}{\left( C_1 \right)^2} \left( 1 + 2\left( T + 1 \right)^{\delta + 1} C^2 \right)^2.$$
By using Gronwall's inequality, one has
$$
\left\| H^k u_{\epsilon}(t) \right\|^2_{L^2(\mathbb{R}^2)} \leq C_{5}\left( 1 + T e^T \right)\left( A^k k! \right)^2.
$$
In view of \eqref{2-10}, we can get that
\begin{equation}\label{2-13}
\begin{aligned}
&\left\| H^k u_{\epsilon}(t) \right\|^2_{L^2(\mathbb{R}^2)} + \frac{7}{4} C_1 \int_0^t \left\| \partial_y H^k u(s) \right\|^2_{L^2(\mathbb{R}^2)} ds \\
&\leq C_{5}\left( A^k k! \right)^2 + \frac{1}{2} TC_{5}\left( 1 + Te^T \right) \left( A^k k! \right)^2 \\
&\le C_{5}(1+T+T^{2}e^{T})\left( A^k k! \right)^2.
\end{aligned}
\end{equation}
Finally, taking 
$$A\ge\max\left\{A_{0}, \left( 2 \left( T + 1 \right)^{\delta + 1}B \right)^2 + 1, (8C(T + 1)^{\delta + 1})^3, \sqrt{C_{5}(1+T+T^{2}e^{T})}\right\},$$
then with \eqref{2-13}, we have
\begin{equation*}
\left\| H^k u_{\epsilon}(t) \right\|^2_{L^2(\mathbb{R}^2)} + \frac{C_1}{2} \int_0^t \left\| \partial_y H^k u_{\epsilon}(s) \right\|^2_{L^2(\mathbb{R}^2)} ds \leq \left( A^{k + 1} k! \right)^2.
\end{equation*}
Since $A$ is independent of $0<\epsilon<1$, by the compactness and uniqueness of the solution, we can obtain that for all $0<t\le T$
\begin{equation*}
\left\| H^k u(t) \right\|^2_{L^2(\mathbb{R}^2)} + \frac{C_1}{2} \int_0^t \left\| \partial_y H^k u(s) \right\|^2_{L^2(\mathbb{R}^2)} ds \leq \left( A^{k + 1} k! \right)^2.
\end{equation*} 
\end{proof}
\bigskip
Now we give the proof of the Theorem \ref{Theorem}.

\textbf{Proof of Theorem \ref{Theorem}}
Let
$$
H_{\delta/2} = \frac{t^{\delta/2 + 1}}{\delta/2 + 1} \partial_x + t^{\delta/2} \partial_y, \quad H_\delta = H,
$$
then $\partial_x$ and $\partial_y$ can be represented as the linear combination of $H_{\delta/2}$ and $H_{\delta}$
$$
\begin{cases}
t^{\delta + 1} \partial_x = -\frac{\left( \delta + 2 \right) \left( \delta + 1 \right)}{ \delta} H_{\delta} +\frac{\left( \delta + 2 \right) \left( \delta + 1 \right)}{\delta} t^{\delta/2} H_{\delta/2}, \\
t^\delta \partial_y =  \frac{2\delta + 2}{\delta} H_{\delta} - \frac{\delta + 2}{ \delta} t^{\delta/2} H_{\delta/2}.
\end{cases}
$$
Then from the Plancherel theorem, one has
\begin{equation*}
\begin{aligned}
& \quad t^{\left( \delta + 1 \right)m} \left\| \partial_x^m u(t) \right\|_{L^2(\mathbb{R}^2)} \\
&=\left\| \mathscr{F} \left[ \left( -\frac{\left( \delta + 2 \right) \left( \delta + 1 \right)}{ \delta} H_{\delta} +\frac{\left( \delta + 2 \right) \left( \delta + 1 \right)}{\delta} t^{\delta/2} H_{\delta/2} \right)^m u(t) \right] \right\|_{L^2(\mathbb{R}^2)} \\
&=\left\| \left( a_1 + a_2 \right)^m \hat{u}(t) \right\|_{L^2(\mathbb{R}^2)},
\end{aligned}
\end{equation*}
where $\mathscr{F}$ is the Fourier transform define via
$$\mathscr{F}u(\xi)=\int_{\mathbb R^{2}}e^{-2\pi ix\cdot\xi}u(x)dx, \quad \forall u\in\mathscr S(\mathbb R^{2}),$$
and $a_{1}, a_{2}$ are the Fourier multipliers
\begin{equation*}
\begin{aligned}
&\mathscr{F} \left( -\frac{\left( \delta + 2 \right) \left( \delta + 1 \right)}{ \delta} H_{\delta} u(t) \right) \left( \xi_1, \xi_2 \right) = a_1 \left( \xi_1, \xi_2 \right) \hat{u}(t, \xi_1, \xi_2),\\
&\mathscr{F} \left( \frac{\left( \delta + 2 \right) \left( \delta + 1 \right)}{\delta} t^{\delta/2} H_{\delta/2} u(t) \right) \left( \xi_1, \xi_2 \right) = a_2 (\xi_1, \xi_2) \hat{u}(t, \xi_1, \xi_2).
\end{aligned}
\end{equation*}
Then by using the Proposition \ref{Proposition 4.1.}, there exists $A_1, A_2>0$ such that
$$
 \left\| H^m_{\delta} u(t) \right\|_{L^2(\mathbb{R}^2)} \le A_1^{m + 1} m!,\ \ \ \ \  \left\| H^m_{\delta/2} u(t) \right\|_{L^2(\mathbb{R}^2)}\le A_2^{m + 1} m!,
$$
it follows that for any $t\in]0, T]$, 
\begin{equation*}
\begin{aligned}
\left\| \left( a_1 + a_2 \right)^m \hat{u}(t) \right\|_{L^2(\mathbb{R}^2)} 
&\leq 2^m \left( \left\| a_1^m \hat{u}(t) \right\|_{L^2(\mathbb{R}^2)} + \left\| a_2^m \hat{u}(t) \right\|_{L^2(\mathbb{R}^2)} \right) \\
&\leq \tilde{C}^m \left( \left\| H^m_{\delta} u(t) \right\|_{L^2(\mathbb{R}^2)} + \left\| H^m_{\delta/2} u(t) \right\|_{L^2(\mathbb{R}^2)} \right) \\
&\leq \tilde{C}^m \left( 2 \tilde{A}^{m + 1} m! \right) = 2 \tilde{A} \left( \tilde{C} \tilde{A} \right)^m m! ,
\end{aligned}
\end{equation*}
with $  \tilde{A}=\max\{A_1, A_2\}$ and $\tilde C$ depends on $\delta$ and $T$, this implies 
\begin{equation}\label{3-1}
\sup_{0 < t \leq T} t^{\left( \delta + 1 \right)m} \left\| \partial_x^m u(t) \right\|_{L^2(\mathbb{R}^2)} \leq 2 \tilde{A}\left( \tilde{C} \tilde{A} \right)^m m!, \quad \forall m \in \mathbb{N}.
\end{equation}
Similar to \eqref{3-1}, we have
\begin{equation}\label{3-2}
\sup_{0 < t \leq T} t^{\delta n} \left\| \partial_y^n u(t) \right\|_{L^2(\mathbb{R}^2)} \leq 2 \tilde{A}\left( \tilde{C} \tilde{A} \right)^n n!, \quad \forall n \in \mathbb{N}.
\end{equation}
Combining the above two inequalities and integration by parts, we can obtain that 
\begin{equation*}
\begin{aligned}
&\sup_{0 < t \leq T} t^{\left( \delta + 1 \right)m + n \delta} \left\| \partial_x^m \partial_y^n u(t) \right\|_{L^2(\mathbb{R}^2)} \\
&\leq \sup_{0 < t \leq T} \left( t^{2\left( \delta + 1 \right)m} \left\| \partial_x^{2m} u(t) \right\|_{L^2(\mathbb{R}^2)} \right)^{\frac{1}{2}} \left( t^{2\delta n} \left\| \partial_y^{2n} u(t) \right\|_{L^2(\mathbb{R}^2)} \right)^{\frac{1}{2}} \\
&\leq \left( 2\tilde{A}  \left( \tilde{C}\tilde{A}  \right)^{2m} \left( 2m \right)! 2\tilde{A}  \left( \tilde{C} \tilde{A}  \right)^{2n} \left( 2n \right)! \right)^{\frac{1}{2}} \\
&\leq 2\tilde{A}  \left( 2 \tilde{C} \tilde{A}  \right)^{m + n} m! n!\leq L^{m + n + 1} m! n!,
\end{aligned}
\end{equation*}
if we take $L \ge 2\tilde{A} (\tilde{C} + 1)$. We have then finished the proof  Theorem \ref{Theorem}.

\bigskip
\noindent {\bf Acknowledgements.}
This work was supported by the NSFC (No.12031006) and the Fundamental
Research Funds for the Central Universities of China.


\begin{thebibliography}{99}  

\bibitem{ref1} J.-L. Chen, W.-X. Li, and C.-J. Xu. Sharp regularization effect for the non-cutoff Boltzmann equation with hard potentials. arXiv e-prints, page arXiv:2305.02861, May 2023.
\bibitem{ref2} A. Pascucci and S. Polidoro. The Moser’s iterative method for a class of ultraparabolic equations. Commun. Contemp. Math., 6(3):395–417, 2004.

\bibitem{k-1} A.N. Kolmogorov. Zuf\"allige Bewegungen. Ann. of Math., 35(2):116-117, 1934.


\bibitem{ref7}  L. H\"ormander. Hypoelliptic second order differential equations. Acta Math., 119:147–171, 1967.
\bibitem{h-1} J. Moser. A Harnack inequality for parabolic differential equations. Comm. Pure Appl. Math., 17(1):101–134, 1964.

\bibitem{I-1} J. Moser. On a pointwise estimate for parabolic differential equations. Comm. Pure Appl. Math., 24(5):727–740, 1971.

\bibitem{ref13} M. Bramanti, M.C. Cerutti and M. Manfredini. $L^p$ estimates for some ultraparabolic operators with discontinuous coefficients. J. Math. Anal. Appl., 200(2):332–354, 1996.
\bibitem{s-1} C. Cinti, A. Pascucci, S. Polidoro. Pointwise estimates for a class of non-homogeneous Kolmogorov equations. Mathematische Annalen, 340:237-264, 2008.
\bibitem{w-1} L. Zhang. The $C^\alpha$ regularity of a class of ultraparabolic equations.  Commun. Contemp. Math., 13(03):375-387, 2011.
\bibitem{ref18} W. Wang, L. Zhang. The $C^\alpha$ regularity of a class of non-homogeneous ultraparabolic equations. Sci. China Ser. A, 52(8):1589-1606, 2009.
\bibitem{p-1} C. Chandrasekhar. Stochastic problems in physics and astronomy. Rew. Modern Phys., 15(1):1-89, 1943.
\bibitem{p-2} N. D. An. Some methods for integrating Fokker-Planck-Kolmogorov equations in the theory of random oscillations. Ukrainian Math. J., 33(1):71-74, 1981.
\bibitem{p-3} A. N. Shiryayev, (ed.). Selected works of A. N. Kolmogorov, vol. \Rmnum{2}, Probablity theory and mathematical statistics. Kluwer Academic Publishers, Dordrecht, 1975.
\bibitem{d-5} F. Anceschi, S. Polidoro and M. A. Ragusa. Moser’s estimates for degenerate Kolmogorov equations with non-negative divergence lower order coefficients. Nonlinear Anal., 189:111568, 2019.
\bibitem{s-2} F. Anceschi, S. Polidoro. A survey on the classical theory for Kolmogorov equation. arXiv preprint, arXiv:1907.05155, 2019.
\bibitem{s-3} E. Lanconelli, S. Polidoro. On a class of hypoelliptic evolution operators. Rend. Sem.
Mat. Univ. Politec. Torino, 52(1):29–63, 1994.
\bibitem{s-4} S. Polidoro and M. A. Ragusa. H\"older regularity for solutions of ultraparabolic equations in divergence form. Potential Anal., 14:341–350, 2001.
\bibitem{t-1} R. Alexandre, Y. Morimoto, S. Ukai, C.-J. Xu, T. Yang. Uncertainty principle and kinetic equations. J. Funct. Anal., 255:2013–2066, 2008.
\bibitem{s-5} M. Manfredini. The Dirichlet problem for a class of ultraparabolic equations. Adv. Diff. Equ.,  2(5):831–866, 1997.
\bibitem{t-2} C.-J. Xu, Y. Xu. Analytic Gelfand-Shilov smoothing effect of fractional Kramers-Fokker-Planck equation. arXiv e-prints, page arXiv:2305.07936, May 2023.
\bibitem{s-6} M. Di Francesco and S. Polidoro. Harnack inequality for a class of degenerate parabolic equations of Kolmogorov type. Adv. Diff. Equ., 11:1261–1320, 2006.
\bibitem{o-2} M. Durand. Régularité Gevrey d’une classe d’opérateurs hypo-elliptiques. J. Math. Pure Appl., 57:323-360, 1978.
\bibitem{o-3}  H. Chen, W.-X. Li, C.-J. Xu. Gevrey hypoellipticity for linear and non-linear Fokker–Planck equations. J. Differential Equations, 246(1):320–339, 2009.
\end{thebibliography}
\end{document}